\documentclass{article}
\usepackage{amsmath}
\usepackage{amsfonts}
\usepackage{amsthm}
\usepackage{amssymb}
\usepackage{fullpage}
\usepackage{graphicx}
\usepackage{scalefnt}
\usepackage{subcaption}
\usepackage{tikz-cd}
\providecommand{\U}[1]{\protect\rule{.1in}{.1in}}
\newtheorem{theorem}{Theorem}

\newtheorem{definition}[theorem]{Definition}

\newtheorem{lemma}[theorem]{Lemma}

\newtheorem{problem}[theorem]{Problem}

\newtheorem{question}[theorem]{Question}

\allowdisplaybreaks

\begin{document}

\title{Almost Bijective Parametrization of \\ $3 \times 3$ Copositive Matrices}
\author{Hoon Hong and Ezra Nance}
\maketitle

\begin{abstract}
        In this work we take a deep dive into the cone of copositive $3 \times 3$ matrices. In doing so we visualize the cone, make geometric observations about it, and prove them. We then use these observations to parametrize the set. In the process we run into issues with surjectivity, and overcome them by resolving singularities and slightly shifting our original approach. We do all of this to ultimately arrive at a novel parametrization of the $3 \times 3$ copositive matrix cone which is surjective and almost injective, which we call almost bijective.
\end{abstract}

\section{Introduction}

For many years positive semi-definite matrices have been studied, that is, matrices whose associated quadratic form is always non-negative. These objects have been interesting and fruitful to look into. They have applications in not only matrix theory, but also in optimization \cite{vandenberghe1996semidefinite} and differential equations \cite{potter1966matrix}. Furthermore, they can be used to answer the quadratic case of the famous Hilbert's 17th problem. It might seem amazing that these matrices have proven to be so useful, but this really shouldn't come as a surprise. The property of being non-negative for all inputs is a very natural one. After all, we live in a world where negative values don't make much sense for many physical quantities such as mass or time. So in the spirit of avoiding negative quantities, \text{copositive} matrices were introduced.

\begin{definition}[Copositive Matrix]
        An $n \times n$ symmetric matrix $A$ is \emph{copositive} if $x^\intercal A x \geq 0$ for all $x \geq 0$.
\end{definition}

Notice that the non-negativity condition has now been extended to the input of the quadratic form, not only the output. Copositive matrices were first introduced by Motzkin \cite{motzkin1952copositive}, and their applications are quite similar to those of positive semi-definite matrices for many of the same reasons \cite{hadeler1983copositive}. While copositive matrices have been studied much less than positive semi-definite ones, quite a bit of research has been done on the implicitization problem \cite{valiaho1986criteria}. That is, given a symmetric matrix how can one decide if it is copositive or not? While we will use some of these results, our goal instead is parametrization. In other words, how do we generate all of the copositive matrices?

\begin{definition}[Cone of Copositive Matrices]
        Let $\mathcal{CF}_{n} \subset \mathbb{R}^{\frac{n(n+1)}{2}}$ be the set of $n \times n$ copositive matrices.
\end{definition}

Most of the results in this area are focused on parametrizing special subsets of the copositive cone, see \cite{dickinson2013irreducible}, \cite{hildebrand2012extreme}, and \cite{afonin2021extreme}. There are more general results involving the parametrization of a general semialgebraic set like \cite{gonzalez1996parametrization}. However, these results would be difficult for us to implement. Furthermore, these generalized results are based off the Cylindrical Algebraic Decomposition (CAD), which yield piecewise results. In this work we wish to avoid piecewise parametrizations as much as possible.

Similar to the positive semi-definite matrices the set of copositive matrices of a particular size form a closed, convex cone. There is a classic result \cite{diananda1962non} which states that for $n < 5$,
\[
        \mathcal{CF}_n = \left\{A+B : A \text{ is positive semi-definite and }B \text{ has non-negative entries}\right\}.
\]
This is a good parametrization of the set, however, we will demand a bit more. Ideally, we would like to parametrize the set injectively, and we can see that using this method we will have $n(n+1)$ parameters for an $\frac{n(n+1)}{2}$-dimensional set. Therefore, we have no hope of this being injective. This is where we will make progress.

As is often the case in mathematics we will begin by considering a small case. When $n=2$ it doesn't take much work to see that
\[
        \mathcal{CF}_2 = \left\{
                \begin{bmatrix}
                        t_1^2 & -t_1t_2 + \lambda \\
                        -t_1t_2 + \lambda & t_2^2
                \end{bmatrix} :
                t_1,t_2,\lambda \geq 0
        \right\}.
\]
Furthermore, this parametrization is injective. This sets the standard for what we want as $n$ goes up. Naturally, we wish to parametrize $\mathcal{CF}_3$ next which brings us to the main problem in the paper.

\begin{problem}
        Find a parametrization map $\Phi: U \to \mathcal{CF}_3$ such that
        \begin{enumerate}
                \item The domain $U$ is an easily described.
                \item The map $\Phi$ is surjective.
                \item The map $\Phi$ is (almost) injective.
        \end{enumerate}
\end{problem}

\noindent We will address this problem by using the geometry of $\mathcal{CF}_3$ to inform our intuition and progress from there.


\section{Visualizing $\mathcal{CF}_{3}$}

In order the visualize $\mathcal{CF}_3 \subset \mathbb{R}^6$, we will need to take a 3-dimensional slice. However, we would like for this slice to be ``generic" in some sense. When studying copositive matrices it is quite common to scale the diagonal elements of the matrix to be 1 since
\[
        \begin{bmatrix}
                a_{11} & a_{12} & a_{13} \\
                a_{12} & a_{22} & a_{23} \\
                a_{13} & a_{23} & a_{33}
        \end{bmatrix} \in \mathcal{CF}_3
        \iff
        \begin{bmatrix}
                1 & \frac{a_{12}}{\sqrt{a_{11}a_{22}}} & \frac{a_{13}}{\sqrt{a_{11}a_{33}}} \\
                \frac{a_{12}}{\sqrt{a_{11}a_{22}}} & 1 & \frac{a_{23}}{\sqrt{a_{22}a_{33}}} \\
                \frac{a_{13}}{\sqrt{a_{11}a_{33}}} & \frac{a_{23}}{\sqrt{a_{22}a_{33}}} & 1
        \end{bmatrix} \in \mathcal{CF}_3.
\]
This works as long as no $a_{ii} = 0$. In other words, this almost always works. With this property in mind we now will define a particular slice of $\mathcal{CF}_3$.
\begin{definition}
        Let $\widehat{\mathcal{CF}}_{3}$ be the subset of $\mathcal{CF}_{3}$ consisting of matrices whose diagonal elements equal 1, that is,
        \[
        \widehat{\mathcal{CF}}_{3} = \left\{ A \in \mathcal{CF}_3 : A =
                \begin{bmatrix}
                        1 & a_{12} & a_{13} \\
                        a_{12} & 1 & a_{23} \\
                        a_{13} & a_{23} & 1
                \end{bmatrix}
        \right\}.
        \]
\end{definition}
\noindent Note that $\widehat{\mathcal{CF}}_3$ is a 3-dimensional, closed, convex cone. Thus, we can visualize it. In Figure \ref{fig:sampledSet} we can see the result of sampling the set.

\begin{figure}[h!]
        \centering
        \includegraphics[scale=0.4]{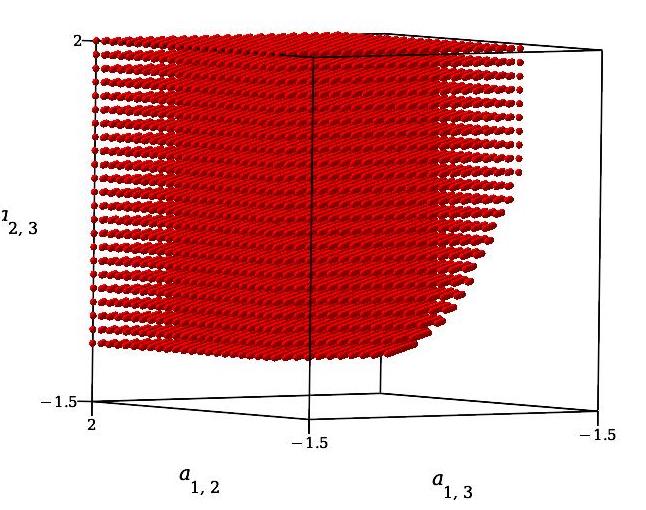}
        \caption{Sampled Points from $\widehat{\mathcal{CF}}_3$}
        \label{fig:sampledSet}
\end{figure}

From visual inspection we can pick out a few notable features:
\begin{enumerate}
        \item The boundary of $\widehat{\mathcal{CF}}_3$ seems to be comprised of 4 discrete pieces: 3 flat pieces and 1 curved piece.
        \item The set appears to be bounded by the cone formed by the inequalities $a_{ij} \geq -1$ for all $i \neq j$.
        \item The curved side looks like it is bounded by the boundary of a 2-dimensional simplex.
\end{enumerate}
These observations are much clearer in Figure \ref{fig:Boundary} where we have color coded the parts of the boundary. The red, blue, and green portions are the flat pieces, and the gray portion is the curved 2-dimensional simplex. It should be noted that the gaps between these pieces are not really there. They are a result of using software to plot the boundary. We can confirm observation 2 in the following lemma.

\begin{figure}[h!]
        \centering
        \includegraphics[scale=0.4]{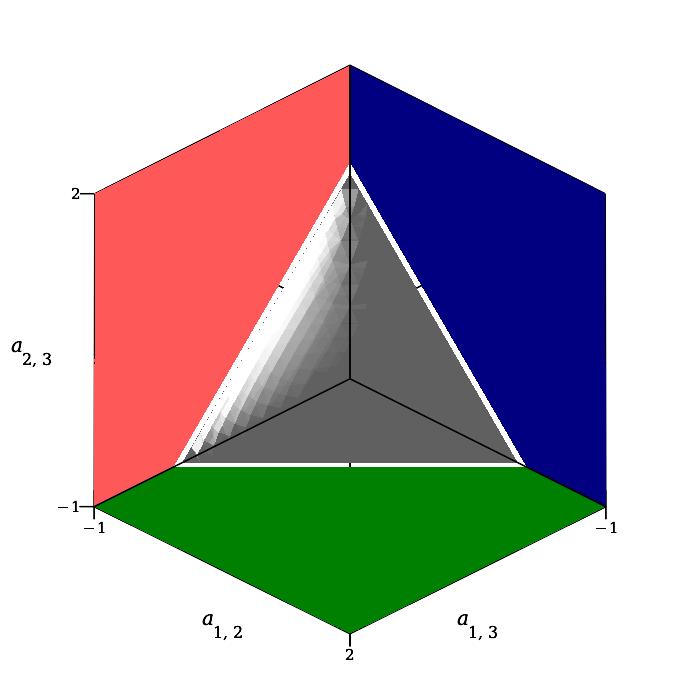}
        \caption{Pieces of the boundary of $\widehat{\mathcal{CF}}_3$}
        \label{fig:Boundary}
\end{figure}

\begin{lemma}\label{lem:bound_on_nondiagonal}
        Let $A \in \widehat{\mathcal{CF}}_3$. Then $a_{12},a_{13},a_{23} \geq -1$.
\end{lemma}

\begin{proof}
        Let $A \in \widehat{\mathcal{CF}}_3$. By way of contradiction assume that $a_{12} < -1$. Thus, $a_{12} = -1 - \varepsilon$ for some $\varepsilon > 0$. Then the associated quadratic form is
        \begin{align*}
                x_1^2 + x_2^2 + x_3^2 + 2a_{12}x_1x_2 + 2a_{13}x_1x_3 + 2a_{23}x_2x_3 &= 
                x_1^2 + x_2^2 + x_3^2 + -2(1+\varepsilon)x_1x_2 + 2a_{13}x_1x_3 + 2a_{23}x_2x_3 \\
                &= (x_1-x_2)^2  + x_3^2 - 2\varepsilon x_1x_2 + 2a_{13}x_1x_3 + 2a_{23}x_2x_3.
        \end{align*}
        Evaluating this polynomial at $x=\left(1,1,0\right)$ gives us $-2\varepsilon < 0$. Therefore, $A \notin \widehat{\mathcal{CF}}_3$, which is a contradiction. By symmetry we also have that $a_{13} \geq -1$ and $a_{23} \geq -1$.
\end{proof}

Lemma \ref{lem:bound_on_nondiagonal} as well as looking at a picture like Figure \ref{fig:Boundary} suggests a method for parametrizing this set. We can take a portion of $\widehat{\mathcal{CF}}_3$ and project it out from the point $(-1,-1,-1)$. From visual inspection it appears that if we take the curved portion of the boundary and project it out, we should be able to recover all of $\widehat{\mathcal{CF}}_3$, see Figure \ref{fig:BoundaryProjection}.

\begin{figure}[h!]
        \centering
        \includegraphics[scale=0.4]{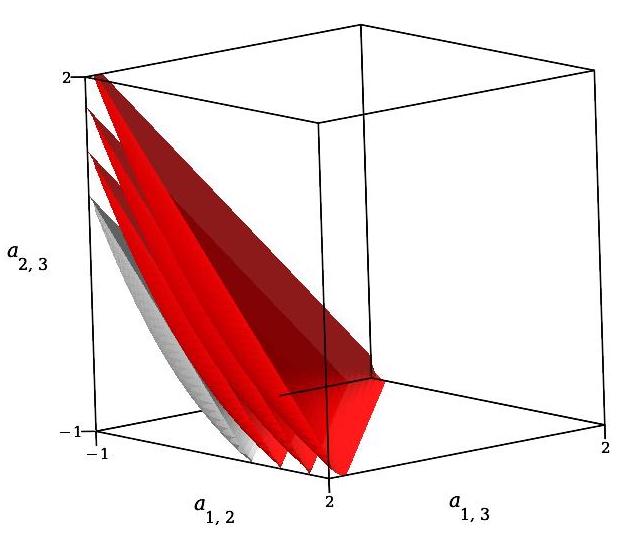}
        \caption{Projection of the curved boundary of $\widehat{\mathcal{CF}}_3$}
        \label{fig:BoundaryProjection}
\end{figure}

\noindent Thus, we have our plan of action. First we will parametrize this curved portion of the boundary of $\widehat{\mathcal{CF}}_3$. Next we will project it out from the point $(-1,-1,-1)$ to get the entirety of $\widehat{\mathcal{CF}}_3$. Finally we will scale the diagonal elements to ``glue" the slices together and form $\mathcal{CF}_3$.

\section{Parametrizing the Curved Boundary of $\widehat{\mathcal{CF}}_3$}

Note that for a matrix to be on the boundary of $\widehat{\mathcal{CF}}_3$, its quadratic form must evaluate to 0 on some non-negative vector. With that in mind, we will begin by trying to figure out what the defining property is for being in this curved boundary. Looking at this set we can pull out a couple of examples.
\[A =
        \begin{bmatrix}
                1 & -1 & -1 \\
                -1 & 1 & 1 \\
                -1 & 1 & 1
        \end{bmatrix}
        \hspace{2em}
        B = 
        \begin{bmatrix}
                1 & -1/2 & -1/2 \\
                -1/2 & 1 & -1/2 \\
                -1/2 & -1/2 & 1
        \end{bmatrix}
\]
We will contrast these with a matrix which is on the boundary but not the curved portion.
\[C =
        \begin{bmatrix}
                1 & -1 & 1 \\
                -1 & 1 & 1 \\
                1 & 1 & 1
        \end{bmatrix}
\]
One can check that all three of these matrices are on the boundary of $\widehat{\mathcal{CF}}_3$, but one difference is that $|A| = |B| = 0$, whereas $|C| \neq 0$. However, being copositive and singular isn't quite enough. Note that the matrix
\[
        \begin{bmatrix}
                1 & 1/2 & 1/2 \\
                1/2 & 1 & 1/2 \\
                1/2 & 1/2 & 1
        \end{bmatrix}
\]
is singular and copositive but is in the interior of $\widehat{\mathcal{CF}}_3$. Analyzing these matrices a little more we can see that not only are $A$ and $B$ singular, but they have a null vector which has non-negative components. This gives us the following lemma.

\begin{lemma}\label{lem:curvedBoundary}
        Let $A \in \partial \widehat{\mathcal{CF}}_3$ with $a_{12} + a_{13} + a_{23} \leq -1$. Then $A$ is singular and has a null vector with non-negative components.
\end{lemma}

\begin{proof}
        Since $A \in \partial\widehat{\mathcal{CF}}_3$, we have $y^\intercal A y = 0$ for some $y \in \mathbb{R}_{\geq 0}\setminus \{0\}$. First we will consider what happens when a component of $y$ is zero. So without loss of generality assume that $y_1 = 0$. Then we have
        \begin{align*}
                y^\intercal A y &= y_2^2 + 2a_{23}y_2y_3 + y_3^2 \\
                &= (y_2-y_3)^2 + 2(a_{23} + 1)y_2y_3.
        \end{align*}
        Note that if $y_2$ or $y_3$ is zero, then we get a contradiction with $y = 0$. Thus, $y_2$ and $y_3$ are both non-zero, and we can solve for $a_{23}$.
        \[
                a_{23} = -\frac{(y_2-y_3)^2}{2y_2y_3} - 1
        \]
        From this we can see that $a_{23} \leq -1$, but from Lemma \ref{lem:bound_on_nondiagonal} we see that $a_{23} \geq -1$. Hence, $a_{23}$ must equal $-1$. To achieve this we must have have $y_2 = y_3 = t > 0$. Furthermore, the condition $a_{12} + a_{13} + a_{23} \leq -1$ becomes the condition $a_{12} + a_{13} \leq 0$. With these in mind we arrive at
        \[
                Ay = \begin{bmatrix}
                        a_{12}y_2 + a_{13}y_3 \\
                        y_2 - y_3 \\
                        -y_2 + y_3
                \end{bmatrix} = 
                t\begin{bmatrix}
                        a_{12} + a_{13} \\
                        0 \\
                        0
                \end{bmatrix}.
        \]
        From here we will rule out the possibility of $a_{12} + a_{13} < 0$. To see this we will consider the quadratic form on a general $x$ and get
        \[
                x_1^2 + (x_2-x_3)^2 + 2a_{12}x_1x_2 + 2a_{13}x_1x_3.
        \]
        Evaluating this on $x = \left( 1,-\frac{1}{a_{12}+a_{13}},-\frac{1}{a_{12}+a_{13}}\right)$ we get the value $1-2 = -1 < 0$. This contradicts $A$ being copositive, and we finish the first case. When $y_1 = 0$, $A$ is singular with a null vector of $(0,t,t) \in \mathbb{R}_{\geq 0 } \setminus \{0\}$. We will have symmetric arguments if a different component of $y$ is zero.
        
        For the next case assume that $y \in \mathbb{R}_{>0}$. Here we will show that $y^\intercal A y = 0$ implies that $Ay = 0$. To proceed we will assume that $y^\intercal A y = 0$ and $Ay \neq 0$, and we will arrive at a contradiction. Since $Ay \neq 0$ there exists some $i \in \{1,2,3\}$ such that the $i$th component of the vector $(Ay)$ is  non-zero. Here we have to split into two cases.
        
        Assume that $(Ay)_i > 0$. Let $\varepsilon \in \mathbb{R}$ be such that $0 < \varepsilon < \min\{2(Ay)_i, y_i\}$. With this setup we have two valuable properties. We have $\varepsilon - 2(Ay)_i < 0$, which implies $\varepsilon^2 - 2\varepsilon(Ay)_i < 0$. Also, we have $y - \varepsilon e_i \in \mathbb{R}_{>0}$. With these in mind we have
        \begin{align*}
                (y-\varepsilon e_i)^\intercal A (y-\varepsilon e_i) &= y^\intercal A y - \varepsilon e_i^\intercal A y - \varepsilon y^\intercal A e_i + \varepsilon^2 e_i^\intercal A e_i \\
                &= -\varepsilon (Ay)_i - \varepsilon (Ay)_i + \varepsilon^2 \\
                &< 0.
        \end{align*}
        This contradicts $A$ being copositive.
        
        For the last case assume that $(Ay)_i < 0$. Let $\varepsilon \in \mathbb{R}$ be such that $0 < \varepsilon < \min\{-2(Ay)_i, y_i\}$. Similar to the last case now we have $\varepsilon^2 + 2\varepsilon(Ay)_i < 0$ and $y + \varepsilon e_i \in \mathbb{R}_{>0}$. Now we calculate the quadratic form
        \begin{align*}
                (y+\varepsilon e_i)^\intercal A (y+\varepsilon e_i) &= y^\intercal A y + \varepsilon e_i^\intercal A y + \varepsilon y^\intercal A e_i + \varepsilon^2 e_i^\intercal A e_i \\
                &= \varepsilon (Ay)_i + \varepsilon (Ay)_i + \varepsilon^2 \\
                &< 0,
        \end{align*}
        and we arrive at our contradiction.
\end{proof}

This leads us to an idea. Since all of the matrices on this curved surface have a non-negative null vector, we can use the coordinates of that null vector as our parameters. If we assume that $p \geq 0$ is a null vector, then we have the following system.
\[
        \begin{bmatrix}
                1 & a_{12} & a_{13} \\
                a_{12} & 1 & a_{23} \\
                a_{13} & a_{23} & 1
        \end{bmatrix}
        \begin{bmatrix}
                p_1 \\ p_2 \\ p_3
        \end{bmatrix} =
        \begin{bmatrix}
                0 \\ 0 \\ 0
        \end{bmatrix}
\]
Luckily, we can uniquely solve this system for $a_{ij}$ as long as $p > 0$. This leads to the following lemma.

\begin{lemma}\label{lem:singular_coefficients}
        Let $p \in \mathbb{R}^3_{>0}$ and $A$ be a $3 \times 3$ real symmetric matrix with 1's along the diagonal. Then $Ap = 0$ if and only if
        \[
        a_{12} = \frac{p_3^2 - (p_1^2 + p_2^2)}{2p_1p_2},
        \hspace{2em}
        a_{13} = \frac{p_2^2 - (p_1^2 + p_3^2)}{2p_1p_3},
        \hspace{2em}
        a_{23} = \frac{p_1^2 - (p_2^2 + p_3^2)}{2p_2p_3}.
        \]
\end{lemma}

\begin{proof}
        We have
        \[
        A =
        \begin{bmatrix}
                1 & a_{12} & a_{13} \\
                a_{12} & 1 & a_{23} \\
                a_{13} & a_{23} & 1
        \end{bmatrix}.
        \]
        Note that $Ap = 0$ if and only if
        \[Ap = 
        \begin{bmatrix}
                p_1 + a_{12}p_2 + a_{13}p_3  \\
                a_{12}p_1 + p_2 + a_{23}p_3 \\
                a_{13}p_1 + a_{23}p_2 + p_3
        \end{bmatrix} = 
        \begin{bmatrix}
                0 \\ 0 \\ 0
        \end{bmatrix}.
        \]
        Thus, we have the system
        \[
        \begin{bmatrix}
                p_2 & p_3 & 0 \\
                p_1 & 0 & p_3 \\
                0 & p_1 & p_2
        \end{bmatrix}
        \begin{bmatrix}
                a_{12} \\ a_{13} \\ a_{23}
        \end{bmatrix} = - 
        \begin{bmatrix}
                p_1 \\ p_2 \\ p_3
        \end{bmatrix}.
        \]
        Since $p \in \mathbb{R}^3_{>0}$ the system has a unique solution, which is
        \[
        a_{12} = \frac{p_3^2 - (p_1^2 + p_2^2)}{2p_1p_2},
        \hspace{2em}
        a_{13} = \frac{p_2^2 - (p_1^2 + p_3^2)}{2p_1p_3},
        \hspace{2em}
        a_{23} = \frac{p_1^2 - (p_2^2 + p_3^2)}{2p_2p_3}.
        \]
\end{proof}

Given Lemma \ref{lem:singular_coefficients} we have a parametrization for $3 \times 3$ real symmetric matrices with 1's along the diagonal with a positive null vector:
\[
        \begin{bmatrix}
                1 & \frac{p_3^2 - (p_1^2 + p_2^2)}{2p_1p_2} & \frac{p_2^2 - (p_1^2 + p_3^2)}{2p_1p_3}\\
                \frac{p_3^2 - (p_1^2 + p_2^2)}{2p_1p_2} & 1 &  \frac{p_1^2 - (p_2^2 + p_3^2)}{2p_2p_3}\\
                \frac{p_2^2 - (p_1^2 + p_3^2)}{2p_1p_3} & \frac{p_1^2 - (p_2^2 + p_3^2)}{2p_2p_3} & 1
        \end{bmatrix}.
\]
However, it's possible that not all of these matrices are copositive. In order to check this, we need a sufficient condition for copositivity. Luckily, we don't need a sufficient condition for a general $3 \times 3$ to be copositive, but we only have to worry about singular matrices with a positive null vector. Under these constraints we have a nice way to check copositivity.

\begin{lemma}\label{lem:singular_copositive_iff_psd}
        Let $A$ be a real symmetric matrix such that $Ap = 0$ for some $p \in \mathbb{R}^n_{>0}$. Then $A \in \mathcal{CF}_{n}$ if and only if $A$ is positive semi-definite.
\end{lemma}

\begin{proof}
        If $A$ is positive semi-definite, then $x^\intercal A x \geq 0$. Thus, $A \in \mathcal{CF}_{n}$. For the other direction assume that $A$ is not positive semi-definite. Then there exists some $q \in \mathbb{R}^3$ such that $q^\intercal A q < 0$. Note that there exists some $\varepsilon > 0$ such that $p + \varepsilon q \in \mathbb{R}^3_{> 0}$. Then we have
        \begin{align*}
                (p + \varepsilon q)^\intercal A (p + \varepsilon q) &= p^\intercal A p + \varepsilon p^\intercal A q + \varepsilon q^\intercal A p + \varepsilon^2 q^\intercal A q \\
                &= \varepsilon^2 q^\intercal A q \\
                &< 0.
        \end{align*}
        Hence $A$ is not copositive.
\end{proof}

So in this special case checking copositivity is the same as checking if the matrix is positive semi-definite. Thus, we can check all of the principal minors. Note that by construction the determinant of the whole matrix is 0, and all of the $1 \times 1$ principal minors are 1. Therefore, we really only need to check the $2 \times 2$ principal minors. Doing so gives us the following.
\begin{align*}
        \begin{vmatrix}
                1 & \frac{p_3^2 - (p_1^2 + p_2^2)}{2p_1p_2} \\
                \frac{p_3^2 - (p_1^2 + p_2^2)}{2p_1p_2} & 1
        \end{vmatrix} &=
        -\frac{1}{4p_1^2p_2^2}(p_1+p_2+p_3)(p_1-p_2+p_3)(p_1+p_2-p_3)(p_1-p_2-p_3) \\
        \begin{vmatrix}
                1 & \frac{p_2^2 - (p_1^2 + p_3^2)}{2p_1p_3} \\
                \frac{p_2^2 - (p_1^2 + p_3^2)}{2p_1p_3} & 1
        \end{vmatrix} &=
        -\frac{1}{4p_1^2p_3^2}(p_1+p_2+p_3)(p_1-p_2+p_3)(p_1+p_2-p_3)(p_1-p_2-p_3) \\
        \begin{vmatrix}
                1 & \frac{p_1^2 - (p_2^2 + p_3^2)}{2p_2p_3} \\
                \frac{p_1^2 - (p_2^2 + p_3^2)}{2p_2p_3} & 1
        \end{vmatrix} &=
        -\frac{1}{4p_2^2p_3^2}(p_1+p_2+p_3)(p_1-p_2+p_3)(p_1+p_2-p_3)(p_1-p_2-p_3)
\end{align*}
The sign of each of these minors is determined by signs of the factors $(p_1+p_2+p_3),$ $(p_1-p_2+p_3),$ $(p_1+p_2-p_3),$ and $(p_1-p_2-p_3)$. Note that $(p_1+p_2+p_3) > 0$ for all $p > 0$. So we arrive at a result. Namely, the singular matrices we have parametrized are copositive if and only if $(p_1-p_2+p_3)(p_1+p_2-p_3)(p_1-p_2-p_3) \leq 0$. This is what we must answer now, which leads into the next lemma.

\begin{lemma}\label{lem:coefficients_cone}
        Let $A$ be a $3 \times 3$ real symmetric matrix with 1's along the diagonal, and let $p$ be a positive null vector of $A$. Then $A \in \widehat{\mathcal{CF}}_{3}$ if and only if 
        \[p \in \operatorname{cone}\left\{e_1+e_2,e_1+e_3,e_2 + e_3  \right\}\setminus \left(\operatorname{span}\{e_1+e_2\} \cup \operatorname{span}\{e_1+e_3\} \cup \operatorname{span}\{e_2+e_3\}\right).\]
\end{lemma}

\begin{proof}
        From the discussion above we know that $A \in \widehat{\mathcal{CF}}_3$ if and only if
        \[
                (p_1-p_2+p_3)(p_1+p_2-p_3)(p_1-p_2-p_3) \leq 0.
        \]
        So we must consider the signs of the following expressions.
        \begin{align}
                (p_1-p_2-p_3)\label{factor1} \\
                (p_1-p_2+p_3)\label{factor2} \\
                (p_1+p_2-p_3)\label{factor3}
        \end{align}
        We now have cases.
        \begin{enumerate}
                \item Assuming that expression (\ref{factor1}), (\ref{factor2}), or (\ref{factor3}) is equal to zero gives us $p_1 = p_2 + p_3$, $p_2 = p_1 + p_3$, or $p_3 = p_1 + p_2$ respectively. In all three of these cases
                \[
                        p \in \operatorname{cone}\left\{e_1+e_2,e_1+e_3,e_2 + e_3  \right\}\setminus \left(\operatorname{span}\{e_1+e_2\} \cup \operatorname{span}\{e_1+e_3\} \cup \operatorname{span}\{e_2+e_3\}\right).
                \]
                \item Assume that expressions (\ref{factor1}), (\ref{factor2}), and (\ref{factor3}) are all less than zero. Then adding the inequalities formed by (\ref{factor2}) and (\ref{factor3}) gives us $2p_1 < 0$, which is a contradiction.
                \item Assume that expression (\ref{factor3}) is less than zero while expressions (\ref{factor1}) and (\ref{factor2}) are greater than zero. Then negating inequality (\ref{factor2}) and adding it to (\ref{factor3}) gives us $2p_2 < 0$, a contradiction.
                \item Assume that expression (\ref{factor2}) is less than zero while expressions (\ref{factor1}) and (\ref{factor3}) are greater than zero. Then negating inequality (\ref{factor1}) and adding it to (\ref{factor2}) gives us $2p_3 < 0$, a contradiction.
                \item Assume that expression (\ref{factor1}) is less than zero while expressions (\ref{factor2}) and (\ref{factor3}) are greater than zero. These inequalities imply that  $p \in \operatorname{cone}\{e_1+e_2,e_1+e_3,e_2+e_3\}$. However, $p$ being strictly positive implies that $p \notin \operatorname{span}\{e_1+e_2\} \cup \operatorname{span}\{e_1+e_3\} \cup \operatorname{span}\{e_2+e_3\}.$
        \end{enumerate}
\end{proof}

With Lemma \ref{lem:coefficients_cone} in hand we can cut down our domain to come from the specified to domain to arrive at a parametrization of the curved boundary. Furthermore, since each expression for $a_{ij}$ is homogeneous and quadratic in the numerator and denominator, the length of the positive null vector chosen has no effect on the copositive matrix generated. Thus, we can assume that the null vector is from a simplex to simplify our domain even further. 

\begin{figure}[h!]
        \centering
        \includegraphics[scale=0.4]{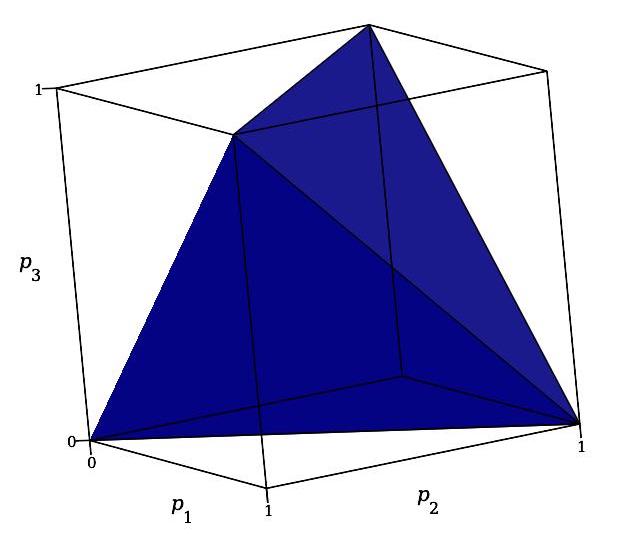}
        \caption{Possible null vectors for $\widehat{\mathcal{CF}}_3$}
        \label{fig:NullSpace}
\end{figure}

\noindent Since our domain now is a portion of a simplex. We will do a change of variables so that it is just a simplex. We will do the change of variables: 
\[p_1 = t_2+t_3, \hspace{1em} p_2 = t_1+t_3, \hspace{1em} p_3=t_1+t_2,\]
and now our domain is the standard simplex, $\Delta_3$, without its corners. Following this change of variables through gives us matrices of the following form.
\[
        \begin{bmatrix}
                1 & \frac{p_3^2 - (p_1^2 + p_2^2)}{2p_1p_2} & \frac{p_2^2 - (p_1^2 + p_3^2)}{2p_1p_3}\\
                \frac{p_3^2 - (p_1^2 + p_2^2)}{2p_1p_2} & 1 &  \frac{p_1^2 - (p_2^2 + p_3^2)}{2p_2p_3}\\
                \frac{p_2^2 - (p_1^2 + p_3^2)}{2p_1p_3} & \frac{p_1^2 - (p_2^2 + p_3^2)}{2p_2p_3} & 1
        \end{bmatrix} \mapsto
        \begin{bmatrix}
                1 & \frac{2t_1t_2}{(1-t_1)(1-t_2) }-1 & \frac{2t_1t_3}{(1-t_1)(1-t_3) }-1 \\
                \frac{2t_1t_2}{(1-t_1)(1-t_2) }-1 & 1 & \frac{2t_2t_3}{(1-t_2)(1-t_3) }-1 \\
                \frac{2t_1t_3}{(1-t_1)(1-t_3) }-1 & \frac{2t_2t_3}{(1-t_2)(1-t_3) }-1 & 1
        \end{bmatrix}
\]
We can formalize this in the following theorem.

\begin{theorem}\label{thm:boundaryParam1}
        The map $\Psi: \operatorname{int}(\Delta_3) \to \{A \in \partial\widehat{\mathcal{CF}}_3 : a_{12} + a_{13} + a_{23} < -1 \}$ defined by
        \[
                (t_1,t_2,t_3) \mapsto
                \begin{bmatrix}
                        1 & \frac{2t_1t_2}{(1-t_1)(1-t_2) }-1 & \frac{2t_1t_3}{(1-t_1)(1-t_3) }-1 \\
                        \frac{2t_1t_2}{(1-t_1)(1-t_2) }-1 & 1 & \frac{2t_2t_3}{(1-t_2)(1-t_3) }-1 \\
                        \frac{2t_1t_3}{(1-t_1)(1-t_3) }-1 & \frac{2t_2t_3}{(1-t_2)(1-t_3) }-1 & 1
                \end{bmatrix}
        \]
        is $\Psi$ bijective.
\end{theorem}

\begin{proof}
        As discussed above, $\Psi$ is a composition. Namely, $\Psi = f \circ g$, where 
        \[g: \operatorname{int}(\Delta_3) \to \operatorname{int}\left(\operatorname{conv}\{e_1+e_2,e_1+e_2,e_2+e_3\}\right)\] 
        is defined by $g(t_1,t_2,t_3) = (t_2+t_3,t_1+t_3,t_1+t_2)$, and 
        \[f: \operatorname{int}\left(\operatorname{conv}\{e_1+e_2,e_1+e_2,e_2+e_3\}\right) \to \{A \in \partial\widehat{\mathcal{CF}}_3 : a_{12} + a_{13} + a_{23} < -1\}\]
        is defined by
        \[
                f(p_1,p_2,p_3) = 
                \begin{bmatrix}
                        1 & \frac{p_3^2 - (p_1^2 + p_2^2)}{2p_1p_2} & \frac{p_2^2 - (p_1^2 + p_3^2)}{2p_1p_3}\\
                        \frac{p_3^2 - (p_1^2 + p_2^2)}{2p_1p_2} & 1 &  \frac{p_1^2 - (p_2^2 + p_3^2)}{2p_2p_3}\\
                        \frac{p_2^2 - (p_1^2 + p_3^2)}{2p_1p_3} & \frac{p_1^2 - (p_2^2 + p_3^2)}{2p_2p_3} & 1
                \end{bmatrix}.
        \]
        It can be easily checked that the linear map $g$ is bijective. Now we will check that $f$ is bijective, completing the proof. Note that $f$ is bijective if all matrices in $\{A \in \partial\widehat{\mathcal{CF}}_3 : a_{12} + a_{13} + a_{23} < -1\}$ can be uniquely determined by a single null vector. We will begin by checking the rank of matrices in the image of $f$. This map $f$ was specifically chosen so that $|f(p)| = 0$. So clearly the rank of $f(p)$ is either 1 or 2. For now let us assume that the rank is 1. Then by following the discussion in the proof of Lemma \ref{lem:coefficients_cone} we see that at least one of the following equations must be true to force the $2 \times 2$ minors to be 0.
        \begin{align*}
                p_1 - p_2 - p_3 &= 0 \\
                p_1 - p_2 + p_3 &= 0 \\
                p_1 + p_2 - p_3 &= 0
        \end{align*}
        This implies $p_1 = p_2 + p_3$, $p_2 = p_1+p_3$, or $p_3 = p_1+p_2$. All of these conditions occur on the boundary of $\operatorname{conv}\{e_1+e_2,e_1+e_3,e_2+e_3\}$, which is not in the domain of $f$. Hence, $f(p)$ is rank 2 for all $p$ in the domain.
        
        This is actually all we need. If $f(p) = f(p')$ with $p \neq p'$. Then because of our choice of domain, $p$ and $p'$ are linearly independent. Hence, $f(p)$ has at least a 2 dimensional null space, but this contradicts $f(p)$ being rank 2. Thus, $f$ is injective. Based on Lemma \ref{lem:coefficients_cone} we know that if the domain of $f$ is expanded to be $\operatorname{conv}\{e_1+e_2,e_1+e_3,e_2+e_3\}\setminus\{e_1+e_2,e_1+e_3,e_2+e_3\}$, then its image would be a superset of $\{A \in \partial\widehat{\mathcal{CF}}_3 : a_{12} + a_{13} + a_{23} < -1\}$. So to check surjectivity we will show that all of the points left out of the domain of $f$ do not map to points in the codomain. From symmetry we only need to see where $f$ maps points which satisfy $p_1 = p_2+p_3$. If we make this substitution we see that
        \[
                f(p_2+p_3,p_2,p_3) =
                \begin{bmatrix}
                        1 & -1 & -1\\
                        -1 & 1 &  1\\
                        -1 & 1 & 1
                \end{bmatrix},
        \]
        which does not satisfy the condition that $a_{12} + a_{13} + a_{23} < -1$. Thus, $f$ is surjective.
        
\end{proof}

\section{Resolving Singularities in the Boundary Parametrization}

The parametrization in Theorem \ref{thm:boundaryParam1} works for almost all points on this curved boundary. Unfortunately, we miss out on most of its limit points. For instance, we cannot obtain the matrix
\[
        \begin{bmatrix}
                1 & -1 & 0 \\
                -1 & 1 & 0 \\
                0 & 0 & 1
        \end{bmatrix}.
\]
Furthermore, the domain in Theorem \ref{thm:boundaryParam1} is slightly awkward. In this section we will fix both of these issues. We will change our map so that it is defined everywhere on $\Delta_3$, and in doing so we will pick up the desired limit points so that we will be parametrizing the closed set $\{A \in \partial\widehat{\mathcal{CF}}_3 : a_{12}+a_{13}+ a_{23} \leq -1 \}$. To do this we will dive a little deeper into the limit points of our parametrization 
\[\phi: \operatorname{conv}\{e_1+e_2,e_1+e_3,e_2+e_3\}\setminus\{e_1+e_2,e_1+e_3,e_2+e_3\} \to \{A \in \partial\widehat{\mathcal{CF}}_3 : a_{12}+a_{13}+ a_{23} \leq -1 \}\]
\[
(p_1,p_2,p_3) \mapsto
\begin{bmatrix}
        1 & \frac{p_3^2 - (p_1^2 + p_2^2)}{2p_1p_2} & \frac{p_2^2 - (p_1^2 + p_3^2)}{2p_1p_3}\\
        \frac{p_3^2 - (p_1^2 + p_2^2)}{2p_1p_2} & 1 &  \frac{p_1^2 - (p_2^2 + p_3^2)}{2p_2p_3}\\
        \frac{p_2^2 - (p_1^2 + p_3^2)}{2p_1p_3} & \frac{p_1^2 - (p_2^2 + p_3^2)}{2p_2p_3} & 1
\end{bmatrix}.
\]

As we saw in Theorem \ref{thm:boundaryParam1}, this map is nicely behaved when we restrict the domain to be just the interior of $\operatorname{conv}\{e_1+e_2,e_1+e_3,e_2+e_3\}$, that is, when the matrix has a null vector with strictly positive components. So now let us consider matrices in $\widehat{\mathcal{CF}}_3$ with null vectors which have 0 components. So we will assume $A \in \widehat{\mathcal{CF}}_3$ and that $Ap = 0$ for some $p \in \mathbb{R}^3_{\geq 0}\setminus\{0\}$. Note that $p$ cannot have two components being 0. So we will arbitrarily choose $p_1$ to be 0 and see the results. In order for $p$ to be a null vector we have
\[
        \begin{bmatrix}
                1 & a_{12} & a_{13} \\
                 a_{12} & 1 & a_{23} \\
                 a_{13} & a_{23} & 1
        \end{bmatrix}
        \begin{bmatrix}
                0 \\
                p_2 \\
                p_3
        \end{bmatrix} = 
        \begin{bmatrix}
                a_{12}p_2 +  a_{13}p_3 \\
                p_2 + a_{23}p_3 \\
                a_{23}p_2 + p_3
        \end{bmatrix} = 
        \begin{bmatrix}
                0 \\ 0 \\ 0
        \end{bmatrix}.
\]
From the last two equations we have that $a_{23} = -1$ and $p_2=p_3 = 1$. Now from the first equation we get $a_{12} + a_{13} = 0$. Also, from Lemma \ref{lem:bound_on_nondiagonal} we know that $a_{12},a_{13} \geq -1$. Putting this all together we now have
\[A =
        \begin{bmatrix}
                1 & t & -t \\
                t & 1 & -1 \\
                -t & -1 & 1
        \end{bmatrix}
\]
for some $-1 \leq t \leq 1$. In other words, requiring the null vector to be $(0,1,1)$ doesn't give us just one matrix in $\{A \in \partial\widehat{\mathcal{CF}}_3 : a_{12}+a_{13}+ a_{23} \leq -1 \}$. Instead, we obtain a 1-dimensional set. This is the important idea. Here we see a single point ``splitting" into infinitely many points. And by symmetry we get similar results if we require a different component of the null vector to be 0. These results are summarized in Figure \ref{fig:SingularCorners}.

\begin{figure}[h!]
        \centering
        \includegraphics[scale=0.6]{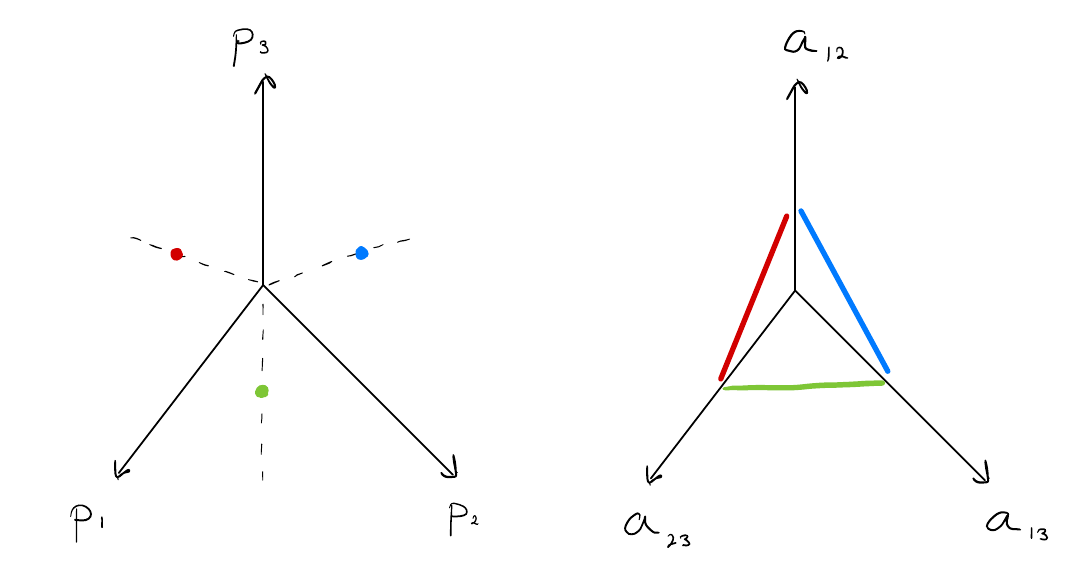}
        \caption{Null vector (left) and corresponding matrices in $\widehat{\mathcal{CF}}_3$ (right)}
        \label{fig:SingularCorners}
\end{figure}

Let's continue to see what happens if we consider other points on the boundary of the domain. So now we will assume that the null vector, $p$, has components $(1,t,1-t)$ for $0<t<1$. In other words $p$ is strictly in between the points $(1,1,0)$ and $(1,0,1)$. Here we get
\[
        \begin{bmatrix}
                1 & a_{12} & a_{13} \\
                a_{12} & 1 & a_{23} \\
                a_{13} & a_{23} & 1
        \end{bmatrix}
        \begin{bmatrix}
                1 \\
                t \\
                1-t
        \end{bmatrix} = 
        \begin{bmatrix}
                (1+a_{13}) + (a_{12}-a_{13})t \\
                (a_{12}+a_{23}) + (1-a_{23})t \\
                (1+a_{13}) - (1-a_{23})t
        \end{bmatrix} = 
        \begin{bmatrix}
                0 \\ 0 \\ 0
        \end{bmatrix}.
\]
Since $t$ is not 0 or 1, we only get a single solution for this system, namely, $(a_{12},a_{13},a_{23}) = (-1,-1,1)$. So here we kind of have the opposite behavior of what we saw before. Now we have infinitely many points ``collapsing" to a single point. Again, by symmetry we have similar behavior for a permutation of the null vector. We can summarize this in Figure \ref{fig:SingularEdges}.

\begin{figure}[h!]
        \centering
        \includegraphics[scale=0.6]{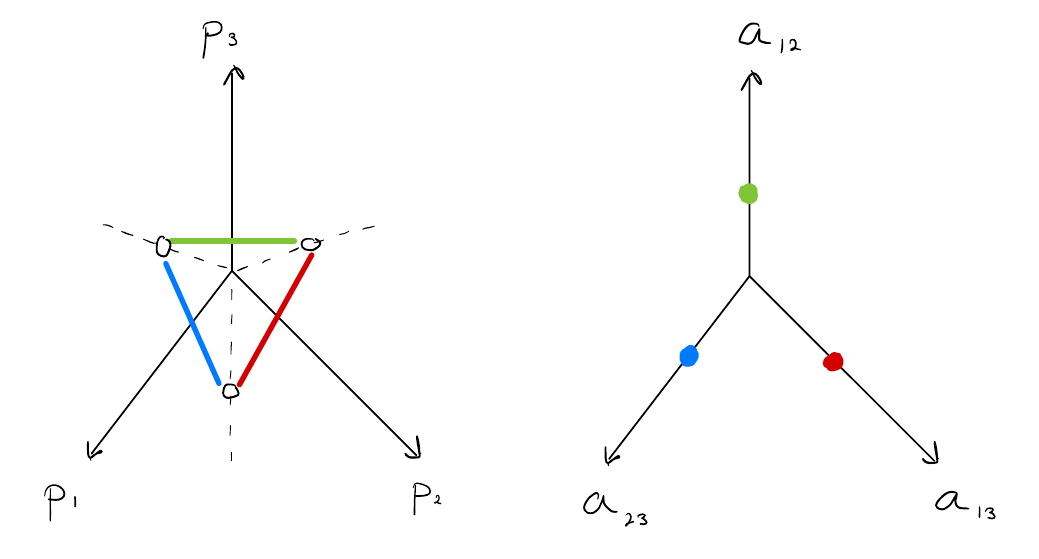}
        \caption{Null vector (left) and corresponding matrix in $\widehat{\mathcal{CF}}_3$ (right)}
        \label{fig:SingularEdges}
\end{figure}

Now that we see the issues we can require our change of coordinates to address them. Not only do we want to do a change of coordinates so that our domain is $\Delta_3$, but we also want this change of coordinates to have singularities which will ``counter act" the behaviors we see in our parametrization. So now we have the problem. Find a map $\psi: \Delta_3 \to \operatorname{conv}\{e_1+e_2,e_1+e_3,e_2+e_3\}$ such that $\psi$ applied to a corner point is singular, and $\psi$ applied to an edge gives us a corner point.
        
This required some searching, but eventually we found the map $\psi: \Delta_3\setminus\{e_1,e_2,e_3\} \to \operatorname{conv}\{e_1+e_2,e_1+e_3,e_2+e_3\}$ defined by
\[
        \psi(t) = \left(\frac{t_3(1-t_1t_2)}{(1-t_1)(1-t_2)},
                                                                \frac{t_2(1-t_1t_3)}{(1-t_1)(1-t_3)},
                                                                \frac{t_1(1-t_2t_3)}{(1-t_2)(1-t_3)}\right).
\]
This map has all of the desired properties. However, we don't have a good explanation of how to construct this map. We simply looked for one which has the outlined properties. This of course leads to questions like, is this the best way to change our coordinates to cancel out the singularities? Is this the only way to change coordinates to achieve this? We don't have the answers to these questions.

With the ultimate goal of having a bijective or almost bijective map in the end we will now check this in the following lemma.

\begin{lemma}\label{lem:changeOfCoords}
        The map $\psi: \operatorname{int}(\Delta_3) \to \operatorname{int}(\operatorname{conv}\{e_1+e_2,e_1+e_3,e_2+e_3\})$ defined by
        \[
        \psi(t) = \left(\frac{t_3(1-t_1t_2)}{(1-t_1)(1-t_2)},
        \frac{t_2(1-t_1t_3)}{(1-t_1)(1-t_3)},
        \frac{t_1(1-t_2t_3)}{(1-t_2)(1-t_3)}\right)
        \]
        is bijective.
\end{lemma}

\begin{proof}
        We will start with injectivity. Assume $\psi(t) = \psi(r)$ for $t,r \in \operatorname{int}(\Delta_3)$. Writing out this condition gives us
        \begin{align*}
                \frac{t_3(1-t_1t_2)}{(1-t_1)(1-t_2)} - \frac{r_3(1-r_1r_2)}{(1-r_1)(1-r_2)} &= 0 \\
                \frac{t_2(1-t_1t_3)}{(1-t_1)(1-t_3)} - \frac{r_2(1-r_1r_3)}{(1-r_1)(1-r_3)} &= 0 \\
                \frac{t_1(1-t_2t_3)}{(1-t_2)(1-t_3)} - \frac{r_1(1-r_2r_3)}{(1-r_2)(1-r_3)} &= 0.
        \end{align*}
        Now if we combine these fractions, set the numerators equal to zero, and reduce the equations using the conditions that $t_1+t_2+t_3 - 1= 0$ and $r_1+r_2+r_3-1 =0$, we get the incredibly long equations
        \begin{align*}
                r_{2}^{2} r_{3} t_{2}^{2}+r_{2}^{2} r_{3} t_{2} t_{3}-r_{2}^{2} t_{2}^{2} t_{3}-r_{2}^{2} t_{2} t_{3}^{2}+r_{2} r_{3}^{2} t_{2}^{2}+r_{2} r_{3}^{2} t_{2} t_{3}-
                r_{2} r_{3} t_{2}^{2} t_{3}-r_{2} r_{3} t_{2} t_{3}^{2}-& \\
                r_{2}^{2} r_{3} t_{2}-r_{2}^{2} r_{3} t_{3} + r_{2}^{2} t_{2} t_{3}- r_{2} r_{3}^{2} t_{2}-r_{2} r_{3}^{2} t_{3}-r_{2} r_{3} t_{2}^{2}+r_{2} t_{2}^{2} t_{3}+r_{2} t_{2} t_{3}^{2}+r_{3} t_{2}^{2} t_{3}+& \\
                r_{3} t_{2} t_{3}^{2}-r_{2}^{2} t_{3}+r_{2} r_{3} t_{2}-r_{2} t_{2} t_{3}+r_{3} t_{2}^{2}+r_{2} t_{3}-r_{3} t_{2} &= 0 \\
                r_{2}^{2} r_{3} t_{2} t_{3}+r_{2}^{2} r_{3} t_{3}^{2}+r_{2} r_{3}^{2} t_{2} t_{3}+r_{2} r_{3}^{2} t_{3}^{2}-r_{2} r_{3} t_{2}^{2} t_{3}-r_{2} r_{3} t_{2} t_{3}^{2}-r_{3}^{2} t_{2}^{2} t_{3}-r_{3}^{2} t_{2} t_{3}^{2}-&\\
                r_{2}^{2} r_{3} t_{2}-r_{2}^{2} r_{3} t_{3}-r_{2} r_{3}^{2} t_{2}-r_{2} r_{3}^{2} t_{3}-r_{2} r_{3} t_{3}^{2}+r_{2} t_{2}^{2} t_{3}+r_{2} t_{2} t_{3}^{2}+r_{3}^{2} t_{2} t_{3}+r_{3} t_{2}^{2} t_{3}+&\\
                r_{3} t_{2} t_{3}^{2}+r_{2} r_{3} t_{3}+r_{2} t_{3}^{2}-r_{3}^{2} t_{2}-r_{3} t_{2} t_{3}-r_{2} t_{3}+r_{3} t_{2} &= 0 \\
                r_{2}^{2} r_{3} t_{2} t_{3}-r_{2} r_{3}^{2} t_{2} t_{3}+r_{2} r_{3} t_{2}^{2} t_{3}+r_{2} r_{3} t_{2} t_{3}^{2}+r_{2}^{2} r_{3} t_{2}+r_{2}^{2} r_{3} t_{3}+r_{2} r_{3}^{2} t_{2}+r_{2} r_{3}^{2} t_{3}- &\\
                r_{2} t_{2}^{2} t_{3}-r_{2} t_{2} t_{3}^{2}-r_{3} t_{2}^{2} t_{3}-r_{3} t_{2} t_{3}^{2}-r_{2}^{2} r_{3}-r_{2} r_{3}^{2}- &\\
                2 r_{2} r_{3} t_{2}-2 r_{2} r_{3} t_{3}+2 r_{2} t_{2} t_{3}+2 r_{3} t_{2} t_{3}+t_{2}^{2} t_{3}+t_{2} t_{3}^{2}+2 r_{2} r_{3}-2 t_{2} t_{3} &= 0.
        \end{align*}
        Using a computer algebra system we get that the solutions to this system are either $(t_1,t_2,t_3) = (r_1,r_2,r_3)$ or $t_1$ is free and $t_2$ is a root of
        \begin{align*}
                \Big(r_{2}^{4} r_{3}^{2}+2 r_{2}^{3} r_{3}^{3}+r_{2}^{2} r_{3}^{4}-r_{2}^{4} r_{3}-3 r_{2}^{3} r_{3}^{2}-2 r_{2}^{2} r_{3}^{3}+r_{2}^{3}+2 r_{2}^{2} r_{3}+r_{2} r_{3}^{2}-r_{2}^{2}-2 r_{2} r_{3}-r_{3}^{2}+r_{2}+r_{3}\Big)z^3 +&\\
                \Big(-r_{2}^{4} r_{3}^{2}-2 r_{2}^{3} r_{3}^{3}-r_{2}^{2} r_{3}^{4}+r_{2}^{4} r_{3}+3 r_{2}^{3} r_{3}^{2}+2 r_{2}^{2} r_{3}^{3}-3 r_{2}^{3} r_{3}-3 r_{2}^{2} r_{3}^{2}-2 r_{2} r_{3}^{2}+2 r_{2}^{2}+5 r_{2} r_{3}+3 r_{3}^{2}-r_{2}-3 r_{3}\Big)z^2 +&\\
                \Big(-r_{2}^{4} r_{3}^{2}-2 r_{2}^{3} r_{3}^{3}-r_{2}^{2} r_{3}^{4}-r_{2}^{4} r_{3}+r_{2}^{3} r_{3}^{2}+2 r_{2}^{2} r_{3}^{3}+3 r_{2}^{3} r_{3}+r_{2}^{2} r_{3}^{2}-r_{2}^{2} r_{3}+r_{2} r_{3}^{2}+r_{2}^{2}-r_{2} r_{3}-2 r_{2}\Big)z +&\\
                \Big(r_{2}^{4} r_{3}^{2}+2 r_{2}^{3} r_{3}^{3}+r_{2}^{2} r_{3}^{4}+r_{2}^{4} r_{3}-r_{2}^{3} r_{3}^{2}-2 r_{2}^{2} r_{3}^{3}-2 r_{2}^{3} r_{3}+r_{2}^{3}+r_{2}^{2} r_{3}-2 r_{2}^{2}\Big).
        \end{align*}
        Again, using a computer algebra system tells us that when $r \in \Delta_3$, the only root that can happen on the interval $[0,1]$ is a root precisely at zero. This tells us if $t_1$ is free, then $t_2 = 0$. But this implies $t$ is on the boundary of $\Delta_3$. Hence, $\psi$ is injective.
        
        Now for surjectivity. Let $r \in \operatorname{int}(\operatorname{conv}\{e_1+e_2,e_1+e_3,e_2+e_3\})$. Then $r = s_1(e_2+e_3) + s_2(e_1+e_3) + s_3(e_1+e_2)$ for some $s \in \operatorname{int}(\Delta_3)$. Now we must find $t \in \operatorname{int}(\Delta_3)$ such that $\psi(t) = r$. So we get the following system.
        \begin{align*}
                \frac{t_3(1-t_1t_2)}{(1-t_1)(1-t_2)} &= s_1+s_2 \\
                \frac{t_2(1-t_1t_3)}{(1-t_1)(1-t_3)} &= s_1+s_3 \\
                \frac{t_1(1-t_2t_3)}{(1-t_2)(1-t_3)} &= s_2+s_3.
        \end{align*}
\end{proof}

\begin{figure}
        \centering
        \begin{subfigure}{.5\textwidth}
                \centering
                \includegraphics[width=\linewidth]{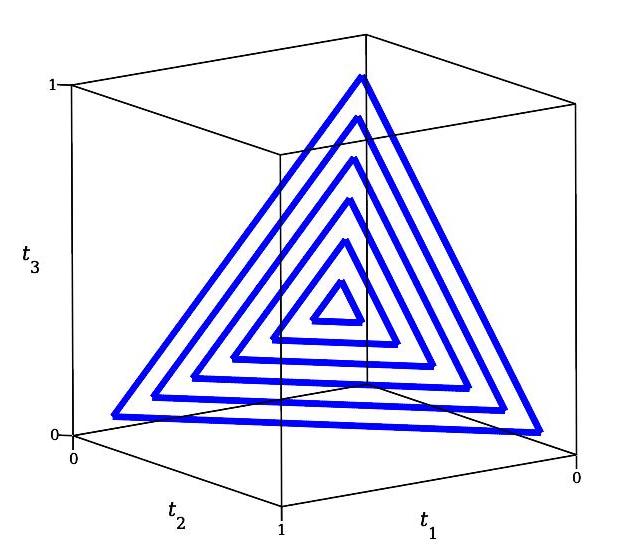}
        \end{subfigure}%
        \begin{subfigure}{.5\textwidth}
                \centering
                \includegraphics[width=\linewidth]{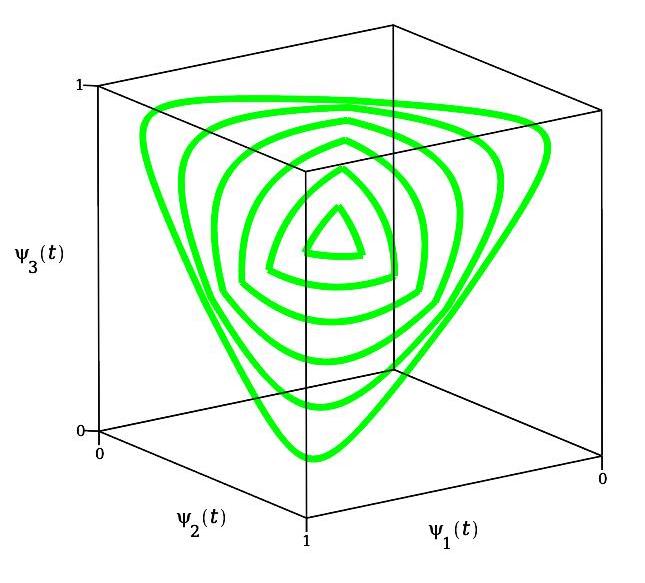}
        \end{subfigure}
        \caption{Standard simplex (left) and $\psi$ evaluated on those curves (right)}
        \label{fig:CoordLoops}
\end{figure}

We can visualize this change of coordinates in Figure \ref{fig:CoordLoops}. Now if we compose this change of coordinates map with our parametrization map, we get the following map $\Psi: \Delta_3 \to \{A \in \partial\widehat{\mathcal{CF}}_3: a_{12}+a_{13}+a_{23} \leq -1 \}$, which leads to the following theorem.

\begin{theorem}\label{thm:BoundaryParam2}
        The map $\Psi: \Delta_3 \to \{A \in \partial\widehat{\mathcal{CF}}_3: a_{12}+a_{13}+a_{23} \leq -1 \}$ defined by
\end{theorem}
\[
        \Psi(t) =
        \begin{bmatrix}
                1 & \frac{2t_3^2(1+t_1)(1+t_2)}{(1-t_1t_3)(1-t_2t_3)} - 1& \frac{2t_2^2(1+t_1)(1+t_3)}{(1-t_1t_2)(1-t_2t_3)} - 1\\
                \frac{2t_3^2(1+t_1)(1+t_2)}{(1-t_1t_3)(1-t_2t_3)} - 1 & 1 & \frac{2t_1^2(1+t_2)(1+t_3)}{(1-t_1t_2)(1-t_1t_3)} - 1\\
                \frac{2t_2^2(1+t_1)(1+t_3)}{(1-t_1t_2)(1-t_2t_3)} - 1 & \frac{2t_1^2(1+t_2)(1+t_3)}{(1-t_1t_2)(1-t_1t_3)} - 1 & 1 
        \end{bmatrix}.
\]
is bijective.

\begin{proof}
        Based on Theorem \ref{thm:boundaryParam1} and Lemma \ref{lem:changeOfCoords} we already have that $\Psi$ restricted to the interior of $\Delta_3$ is injective. Furthermore, the image of this restricted map is the interior of the codomain. Thus, this restriction of $\Psi$ is bijective. Now we must see where the boundaries of the domain get mapped. To see this let us consider an edge of the domain where $t = (s,1-s,0)$. Then we get
        \[
                \Psi(t) =
                \begin{bmatrix}
                        1 & -1 & \frac{(2s-1)(s^2-s-1)}{s^2-s+1} \\
                        -1 & 1 & - \frac{(2s-1)(s^2-s-1)}{s^2-s+1} \\
                         \frac{(2s-1)(s^2-s-1)}{s^2-s+1} &  -\frac{(2s-1)(s^2-s-1)}{s^2-s+1} & 1
                \end{bmatrix}.
        \]
        Note that in this case $a_{12} + a_{13} + a_{23} = -1$. Furthermore, this expression
        \[
                \frac{(2s-1)(s^2-s-1)}{s^2-s+1} 
        \]
        bijectively maps the interval $[0,1]$ to the interval $[-1,1]$. So we have bijectively mapped one edge of the domain to one edge of the codomain. Now by symmetry we can argue that the others behave the same.
\end{proof}

\section{Projecting to Parametrize $\widehat{\mathcal{CF}}_3$}

Now that we have a nice parametrization for the set $\{A \in \partial\widehat{\mathcal{CF}}_3 : a_{12} + a_{13} + a_{23} \leq -1\}$, we can extend this set to recover all of $\widehat{\mathcal{CF}}_3$. As we described before we will do this by projecting the curved boundary from the point $(a_{12},a_{13},a_{23})=(-1,-1,-1)$. We can see a 2-dimensional version of this idea in Figure \ref{fig:ProjIdea}. This leads to the following lemma.

\begin{figure}[h!]
        \centering
        \includegraphics[scale=0.7]{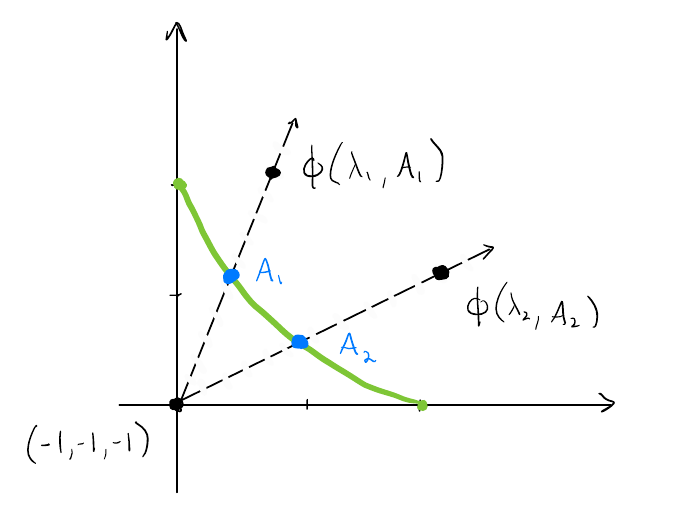}
        \caption{Projecting the boundary to the whole set}
        \label{fig:ProjIdea}
\end{figure}

\begin{lemma}
        The map $\phi: \mathbb{R}_{\geq 0} \times \left\{A \in \partial\widehat{\mathcal{CF}}_3 : a_{12} + a_{13}+a_{23} \leq -1\right\} \to \widehat{\mathcal{CF}}_3$ defined by
        \[
                \phi(\lambda,A) = A + \lambda\left(A -
                \begin{bmatrix}
                        1 & -1 & -1 \\
                        -1 & 1 & -1 \\
                        -1 & -1 & 1
                \end{bmatrix}\right) = 
                \begin{bmatrix}
                        1 & (1+\lambda)a_{12} + \lambda & (1+\lambda)a_{13} + \lambda \\
                        (1+\lambda)a_{12} + \lambda & 1 & (1+\lambda)a_{23} + \lambda \\
                        (1+\lambda)a_{13} + \lambda & (1+\lambda)a_{23} + \lambda & 1
                \end{bmatrix}
        \]
        is a bijection.
\end{lemma}

\begin{proof}
        Note that $\phi$ is injective since there is no way for two distinct rays from $(a_{12},a_{13},a_{23})=(-1,-1,-1)$ to intersect $\left\{A \in \partial\widehat{\mathcal{CF}}_3 : a_{12} + a_{13}+a_{23} \leq -1\right\}$ in two different points and eventually meet up. Also, $\phi$ is surjective since every point in $\widehat{\mathcal{CF}}_3$ can be traced back to a point in $\left\{A \in \partial\widehat{\mathcal{CF}}_3 : a_{12} + a_{13}+a_{23} \leq -1\right\}$ when following a line back towards $(-1,-1,-1)$.
\end{proof}


\section{Scaling to Parametrize $\mathcal{CF}_3$}

Now that we have a good parametrization for $\widehat{\mathcal{CF}}_3$, we can scale the diagonal to obtain all of $\mathcal{CF}_3$. In other words we can extend our chosen slice of $\mathcal{CF}_3$ to \textit{almost} the whole set with the following method. Let $A \in \widehat{\mathcal{CF}}_3$, then we can scale $A$ to have whatever diagonal we wish via the map
\[
\left(
\begin{bmatrix}
        s_1 \\ s_2 \\ s_3
\end{bmatrix},
\begin{bmatrix}
        1 & a_{12} & a_{13} \\
        a_{12} & 1 & a_{23} \\
        a_{13} & a_{23} & 1
\end{bmatrix}
\right) \mapsto
\begin{bmatrix}
        s_1^2 & s_1s_2a_{12} & s_1s_3a_{13} \\
        s_1s_2a_{12} & s_2^2 & s_2s_3a_{23} \\
        s_1s_3a_{13} & s_2s_3a_{23} & s_3^2 
\end{bmatrix},
\]
where $s \in \mathbb{R}^3_{\geq 0}$. This will work for all but a measure zero set in $\mathcal{CF}_3$. So if we wanted to we could stop here. However, we will try to dive a little deeper to examine the problem and see if we can fix it. 

Just like some of our other issues while developing this parametrization, we run into difficulties when zeros occur, but this time we are concerned with zeros in the diagonal. To highlight the problem let's look at the matrices
\[
        A = \begin{bmatrix}
                0 & 1 & 1 \\
                1 & 1 & 1 \\
                1 & 1 & 1
        \end{bmatrix}
        \hspace{2em}
        \text{and}
        \hspace{2em}
        B = \begin{bmatrix}
                0 & 0 & 0 \\
                0 & 1 & 1 \\
                0 & 1 & 1
        \end{bmatrix}.
\]
Clearly both $A$ and $B$ are copositive, but with the current method of scaling the diagonal we can only obtain $B$. This is happening because when we scale a diagonal element, we have to scale an entire row and column as well. This is an issue.

On the road to fixing this problem let's try to reverse the order of scaling and extending. In other words, right now we are parametrizing a portion of the boundary of $\widehat{\mathcal{CF}}_3$, extending that to all of $\widehat{\mathcal{CF}}_3$, and then scaling the diagonal to get $\mathcal{CF}_3.$ Instead let's try to take the curved portion of the boundary of $\widehat{\mathcal{CF}}_3$, scale the diagonal, and then extend the set to get $\mathcal{CF}_3$. Algebraically nothing changes when we do this, we still end up with the map $\Phi:\mathbb{R}^4_{\geq 0} \times \left\{A \in \partial\widehat{\mathcal{CF}}_3: a_{12}+a_{13}+a_{23} \leq -1 \right\} \to \mathcal{CF}_3$
\[
        (s_1,s_2,s_3,\lambda,A) \mapsto
        \begin{bmatrix}
                s_1^2 & s_1s_2(1+\lambda)a_{12} + s_1s_2\lambda & s_1s_3(1+\lambda)a_{13} + s_1s_3\lambda \\
                s_1s_2(1+\lambda)a_{12} + s_1s_2\lambda & s_2^2 & s_2s_3(1+\lambda)a_{23} + s_2s_3\lambda \\
                s_1s_3(1+\lambda)a_{13} + s_1s_3\lambda & s_2s_3(1+\lambda)a_{23} + s_2s_3\lambda & s_3^2
        \end{bmatrix}.
\]
However, thinking about the process in this order will help us fix some of the issues. Namely, the issue that this map is not surjective.

\begin{figure}[h]
        \centering
        \begin{tabular}{lll}
                \includegraphics[width=.3\linewidth]{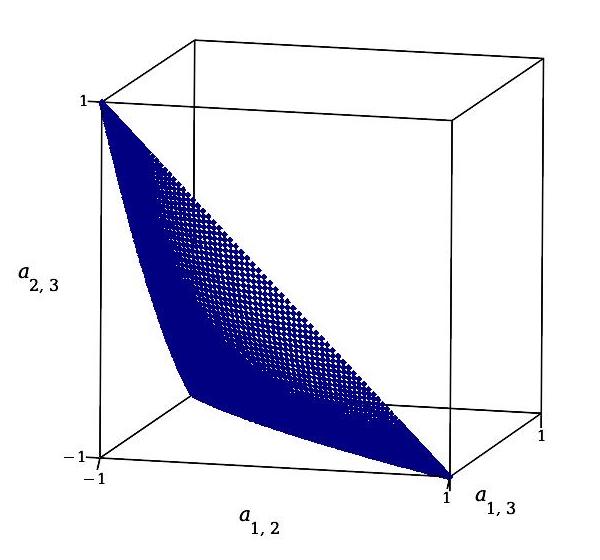} & \includegraphics[width=.3\linewidth]{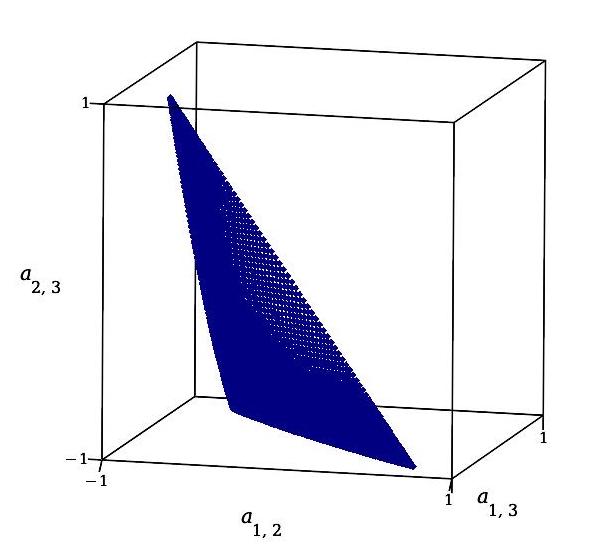} & \includegraphics[width=.3\linewidth]{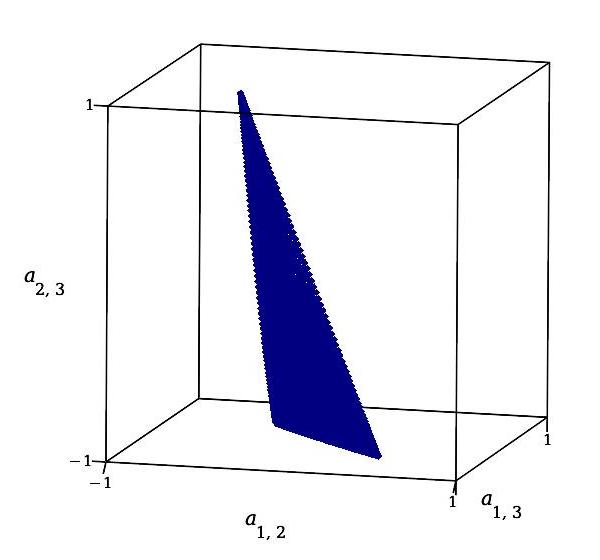}\\
                \includegraphics[width=.3\linewidth]{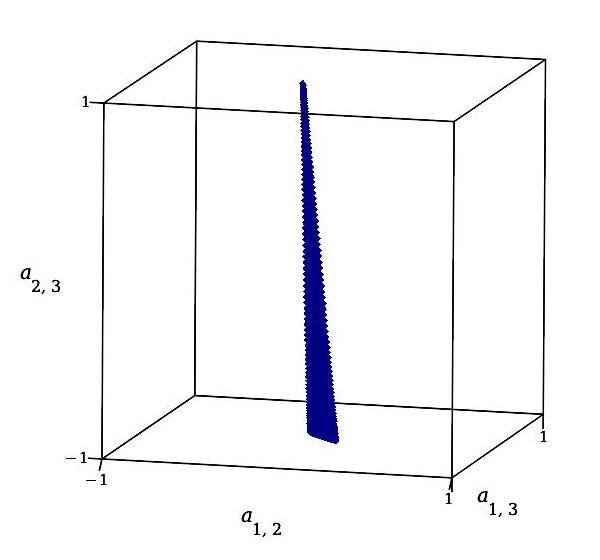} & \includegraphics[width=.3\linewidth]{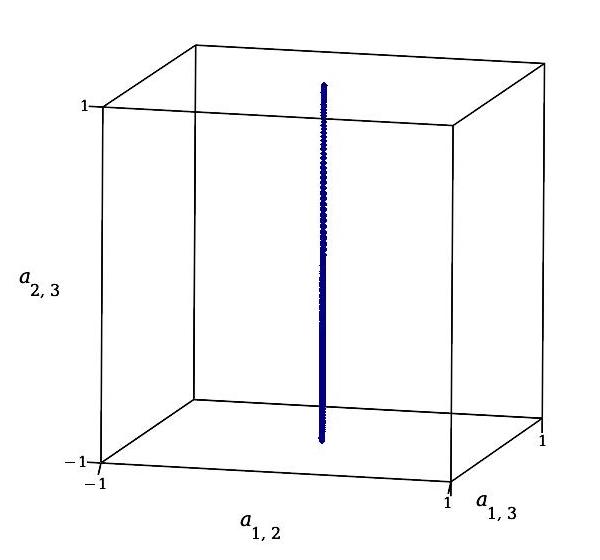} &
        \end{tabular}
        \caption{The off-diagonal elements as a diagonal element changes from 1 to 0}
        \label{fig:ChangingBoundary}
\end{figure}

To try to visualize the problem, we will set $s_2=s_3=1$, but we will allow $s_1$ to decrease from 1 to 0. See Figure \ref{fig:ChangingBoundary}. There are two big issues here. One is that when $s[1]=0$ this boundary collapses into a 1-dimensional object, and so we lose injectivity. Unfortunately, we can't do much about this issue. The second problem is that the point we need to project from, $(-s_1s_2,-s_1s_3,-s_2s_3)$, lies directly under this vertical line when $s_1 = 0$. So projecting will only give us a vertical line, when the copositive set here is everything in a non-negative direction from this line. This issue we can address, and we will do so by projecting a ray \emph{towards} a non-negative direction instead of \emph{from} a point.

To achieve this we will take the curved simplex that we already have which has vertices at $(s_1s_2,-s_1s_3,-s_2s_3)$, $(-s_1s_2,s_1s_3,-s_2s_3)$, and $(-s_1s_2,-s_1s_3,s_2s_3)$, and we will project towards one with vertices at $(s_1s_2+1,-s_1s_3,-s_2s_3)$, $(-s_1s_2,s_1s_3+1,-s_2s_3)$, $(-s_1s_2,-s_1s_3,s_2s_3+1)$. We can do this without introducing any new parameters since this surface we are projecting towards solely depends on the boundary that we've already parametrized. This surface is given by
\begin{align*}
        a_{12} &= (2s_1s_2+1)\frac{t_3^2(1+t_1)(1+t_2)}{(1-t_1t_3)(1-t_2t_3)} - s_1s_2\\
        a_{13} &= (2s_1s_3+1)\frac{t_2^2(1+t_1)(1+t_3)}{(1-t_1t_2)(1-t_2t_3)} - s_1s_3 \\
        a_{23} &= (2s_2s_3+1)\frac{t_1^2(1+t_2)(1+t_3)}{(1-t_1t_2)(1-t_1t_3)} - s_2s_3.
\end{align*}

Projecting towards this surface gives us our final version of the parametrization: $\Phi:\mathbb{R}^4_{\geq 0} \times \Delta_3 \to \mathcal{CF}_3$, where
$(s_1,s_2,s_3,\lambda,t_1,t_2,t_3)$ gets mapped to the matrix
\[
        \begin{bmatrix}
                s_1^2 & (2s_1s_2+\lambda)\frac{t_3^2(1+t_1)(1+t_2)}{(1-t_1t_3)(1-t_2t_3)} - s_1s_2
                & (2s_1s_3+\lambda)\frac{t_2^2(1+t_1)(1+t_3)}{(1-t_1t_2)(1-t_2t_3)} - s_1s_3 \\
                (2s_1s_2+\lambda)\frac{t_3^2(1+t_1)(1+t_2)}{(1-t_1t_3)(1-t_2t_3)} - s_1s_2 & s_2^2 &
                (2s_2s_3+\lambda)\frac{t_1^2(1+t_2)(1+t_3)}{(1-t_1t_2)(1-t_1t_3)} - s_2s_3 \\
                (2s_1s_3+\lambda)\frac{t_2^2(1+t_1)(1+t_3)}{(1-t_1t_2)(1-t_2t_3)} - s_1s_3 & 
                (2s_2s_3+\lambda)\frac{t_1^2(1+t_2)(1+t_3)}{(1-t_1t_2)(1-t_1t_3)} - s_2s_3 & s_3^2
        \end{bmatrix}.
\]

\begin{theorem}
        This map $\Phi: \mathbb{R}^4_{\geq 0} \times \Delta_3 \to \mathcal{CF}_3$ is surjective and almost injective.
\end{theorem}

\begin{proof}
        Let $A \in \mathcal{CF}_3$. The diagonal of $A$ uniquely determines $s_1$, $s_2$, and $s_3$. Let's first consider the case where all $s_i > 0$. In this case the problem simplifies to considering the map $\widehat{\Phi}:\mathbb{R}_{\geq 0} \times \Delta_3 \to \widehat{\mathcal{CF}}_3$, where $(\lambda,t_1,t_2,t_3)$ maps to
        \[
                \begin{bmatrix}
                        1 & (2+\frac{\lambda}{s_1s_2})\frac{t_3^2(1+t_1)(1+t_2)}{(1-t_1t_3)(1-t_2t_3)} - 1
                        & (2+\frac{\lambda}{s_1s_3})\frac{t_2^2(1+t_1)(1+t_3)}{(1-t_1t_2)(1-t_2t_3)} - 1 \\
                        (2+\frac{\lambda}{s_1s_2})\frac{t_3^2(1+t_1)(1+t_2)}{(1-t_1t_3)(1-t_2t_3)} - 1 & 1 &
                        (2+\frac{\lambda}{s_2s_3})\frac{t_1^2(1+t_2)(1+t_3)}{(1-t_1t_2)(1-t_1t_3)} - 1 \\
                        (2+\frac{\lambda}{s_1s_3})\frac{t_2^2(1+t_1)(1+t_3)}{(1-t_1t_2)(1-t_2t_3)} - 1 & 
                        (2 +\frac{\lambda}{s_2s_3})\frac{t_1^2(1+t_2)(1+t_3)}{(1-t_1t_2)(1-t_1t_3)} - 1 & 1
                \end{bmatrix}.
        \]
        Note that $\lambda$ separates the image of $\widehat{\Phi}$ into distinct layers. Thus, $\lambda$ is uniquely determined by $A$ as well. For each $\lambda$ we have a map which is essentially the same as the bijective map in Theorem \ref{thm:BoundaryParam2}. The only difference is that it is a curved simplex with corner points at $(1+\frac{\lambda}{s_1s_1},-1,-1)$, $(-1,1+\frac{\lambda}{s_1s_3},-1)$, and $(-1,-1,1+\frac{\lambda}{s_2s_3})$ instead of at $(1,-1,-1)$, $(-1,1,-1)$, and $(-1,-1,1)$. Hence, each layer is a bijection, and we have achieved surjectivity. Also, this means that our original map $\Phi$ is injective as long as we avoid the measure-zero set of $s_1 = 0$, $s_2=0$, or $s_3=0$.
        
        Now we will consider the case when $s_1 = 0$. This leads to $A$ having the form
        \[
        \begin{bmatrix}
                0 & \lambda\frac{t_3^2(1+t_1)(1+t_2)}{(1-t_1t_3)(1-t_2t_3)}
                & \lambda\frac{t_2^2(1+t_1)(1+t_3)}{(1-t_1t_2)(1-t_2t_3)} \\
                \lambda\frac{t_3^2(1+t_1)(1+t_2)}{(1-t_1t_3)(1-t_2t_3)} & s_2^2 &
                (2s_2s_3+\lambda)\frac{t_1^2(1+t_2)(1+t_3)}{(1-t_1t_2)(1-t_1t_3)} - s_2s_3 \\
                \lambda\frac{t_2^2(1+t_1)(1+t_3)}{(1-t_1t_2)(1-t_2t_3)} & 
                (2s_2s_3+\lambda)\frac{t_1^2(1+t_2)(1+t_3)}{(1-t_1t_2)(1-t_1t_3)} - s_2s_3 & s_3^2
        \end{bmatrix}.
        \]
        Here we lose injectivity. Since $\lambda = 0$ causes the first row and columns to be zero regardless of what $(t_1,t_2,t_3)$ is. However, this map still surjects onto matrices in $\mathcal{CF}_3$ with their top-right entry being 0. Notice that the $2 \times 2$ bottom-right corner is
        \[
                \begin{bmatrix}
                        s_2^2 &
                        (2s_2s_3+\lambda)\frac{t_1^2(1+t_2)(1+t_3)}{(1-t_1t_2)(1-t_1t_3)} - s_2s_3 \\
                        (2s_2s_3+\lambda)\frac{t_1^2(1+t_2)(1+t_3)}{(1-t_1t_2)(1-t_1t_3)} - s_2s_3 & s_3^2
                \end{bmatrix} = 
                \begin{bmatrix}
                        s_2^2 & - s_2s_3 + \gamma \\
                        -s_2s_3 +\gamma & s_3^2
                \end{bmatrix},
        \]
        where $\gamma \in [0,\infty)$. This matches our description of matrices in $\mathcal{CF}_2$. Now for the $(1,2)$ and $(1,3)$ entries, they can be any non-negative numbers. And note, that's exactly what we have here. We can think of $t$ as choosing a non-negative direction, and $\lambda$ extends that direction to achieve any non-negative numbers we wish. By symmetry we can extend this argument, and we have surjectivity in this case.
\end{proof}


\section{Conclusion}

After all of this work we arrived at a parametrization of $\mathcal{CF}_3$ which is quite nice in a lot of ways. Not only is it almost injective, but we can see a lot of the symmetries that exist within $\mathcal{CF}_3$. Of course the goal is to keep going, and to parametrize $\mathcal{CF}_4$ and so on. But we very quickly run into a problem. If we were to try to directly apply the same method to the $n=4$ case, we would get stuck. There is no equivalent result for Lemma \ref{lem:singular_coefficients}. In fact, we were only able to get Lemma \ref{lem:singular_coefficients} because the number of non-diagonal elements, $\binom{n}{2}$, happens to match to the size of the matrix when $n = 3$. This brings us to the following question.
\begin{question}
        For $n > 3$ how do we parametrize the set
        \[
                \left\{A \in \widehat{\mathcal{CF}}_n :  \underset{p\geq 0}{\exists} Ap = 0 \right\}
        \]
        in an injective or almost injective way?
\end{question}

We believe answering this question will make great progress towards the parametrization of all copositive matrices. Another possible direction to go is consider higher degree homogeneous polynomials and try to parametrize the more general \emph{copositive forms}. These are being considered in physics applications \cite{chen2018copositive}. Regardless, there we believe that this is a rich area of study which still has many unknowns. With this paper our goal was to start the ball rolling on parametrizing copositive forms with the smallest non-trivial case, and hopefully we developed ideas which can be utilized for further parametrization.


\bigskip \noindent \textbf{Acknowledgements.}
Hoon Hong was partially supported by US National Science Foundation NSF-CCF-2212461.

\bibliographystyle{plain}
\bibliography{biblio}


\end{document}